\documentclass[12pt,leqno]{amsart}
\usepackage{latexsym,amssymb,amsmath, multicol, rotating,lscape}
\usepackage{letltxmacro}

\usepackage{tikz}
\usetikzlibrary{calc,decorations.markings}

 2

\newcommand{\medmatrix}[4]{\Bigl(\begin{matrix} #1 \!& \!#2 \\[-4pt] #3 \!&\! #4 \end{matrix}\Bigr)}

\newtheorem{thm}{Theorem}[section]
\newtheorem{lem}[thm]{Lemma}

\newtheorem{defn}[thm]{Definition}
\newcommand{\thmref}[1]{Theorem~\ref{#1}}
\newcommand{\lemref}[1]{Lemma~\ref{#1}}

\theoremstyle{remark}

\newenvironment{acknowledgements}{\bigskip\textbf{Acknowledgements.}}{}
\begin{document}

\title[Shifted convolution sum associated to half-integral weight cusp forms]
{\footnotesize{Estimates for the shifted convolution sum involving Fourier coefficients of cusp forms of half-integral weight}}

\author{Abhash Kumar Jha and Lalit Vaishya}

\address{Indian Institute of Technology(Banaras Hindu University) Varanasi - 221005, India.}
\email{abhashkumarjha@gmail.com, abhash.mat@iitbhu.ac.in}

\address{Harish-Chandra Research Institute, HBNI, Chhatnag Road, Jhunsi, Prayagraj - 211019, India.}
\email{lalitvaishya@gmail.com, lalitvaishya@hri.res.in}

\subjclass[2010]{11F30, 11F37}
\keywords{Shifted convolution sum, cusp forms of half-integral weight, Poisson-Voronoi summation formula}

\date{\today}
 
\maketitle

\begin{abstract}

In this article, we obtain certain estimate for the shifted convolution sum involving the Fourier coefficients of half-integral weight cusp forms. 
\end{abstract}

\section{Introduction}
The estimates for the shifted convolution sums involving the Fourier coefficients of automorphic forms have been investigated by several authors. Selberg \cite{Selberg} started the study of shifted convolution sums and obtained the analytic properties of the $L$-function associated with cusp forms. Goldfeld \cite{Goldfeld}, by using the analytic and arithmetic property of Poincar{\'e} series obtained an estimate of the following shifted convolution sum:
$$\displaystyle{\sum_{n\ge 1} a_{f}(n) \overline{a_{g}(n+m)} e^{-n/X}}, $$  where $f$ and $g$ are cusp forms of weight $k$ and $l,$ respectively for the full modular group, and $a_f(n)$ (respectively $a_g(n)$) denotes the $n$-th Fourier coefficient of $f$ (respectively $g$).

 Hafner \cite{Hafner} extended the result of Goldfeld \cite{Goldfeld} for congruence subgroups by using spectral decomposition method. Similar sums have also been considered for other kinds of automorphic forms, see the list \cite{Blomer, Abhash-lalit, Lu, munshi, pitt, saurabh, sun}. Recently, Luo \cite{LWenZhi}  obtained an estimate for the following shifted convolution sum:
\begin{equation}\label{LWZ}
  S(f, g, b):=\sum_{n\ge 1}a_f(n+b)a_g(n)G(n),
  \end{equation}
  where $f$ is a cusp form of weight $k+\frac{1}{2}$ for the group $\Gamma_0(4N), ~g$ is a cusp form of weight $l$ or a Maass cusp form for the group $\Gamma_0(1),$ and $G$ is a smooth function with the support in $[\frac{X}{2}, \frac{5X}{2}]$ satisfying $G^{(p)} (x)\ll (\frac{X}{P})^{-p}$ for all integer $p \ge 0$, where $P$ is a real number with $1 \le P \le X.$ The aim of this paper is to obtain an estimate for $S(f, g, b)$ when $f$ and $g$ are both half-integral weight cusp forms by using a similar method as in  \cite{LWenZhi}. Now, we state the main result of the paper.  
  
  Let $f$ be a cusp form of weight $k+\frac{1}{2}$ and level $4N$ with  Fourier series expansion 
\begin{equation} \label{f-form}
\begin{split}
 \quad f(\tau) & = \displaystyle{\sum_{n \ge 1} a_{f}(n) n^{k/2 - 1/4} e(n\tau)},\end{split}
\end{equation}
and   $g$ be a newform of weight $l+\frac{1}{2}$ and level $4N$  with Fourier series expansion
\begin{equation} \label{g-form}
\begin{split}
 \quad g(\tau) & = \displaystyle{\sum_{n \ge 1} a_{g} (n) n^{l/2 - 1/4} e(n\tau)}.
 \end{split}
\end{equation}
 For a fixed positive integer $b$ and a smooth function $G(x)$  as in \eqref{LWZ}, we consider the following sum:   
 \begin{equation}
\begin{split}
S(f, g, b) \quad & = \quad  \sum_{n \ge 1} a_{f}(n+b) a_{g}(n) G(n). 
\end{split}
\end{equation}
\begin{thm} \label{Half wt form}
Let $f,g$ and $G$ be as above, and $N$ be an odd and squarefree positive integer. Then, we have
\begin{equation}
\begin{split}
S(f, g, b) \quad & = \quad  \sum_{n \ge 1} a_{f}(n+b) a_{g}(n) G(n) \quad \ll_{\epsilon, f, g, b, G} \quad X^{\frac{3}{4} +\epsilon} P^{\frac{3}{2}}. 
\end{split}
\end{equation}
\end{thm}

\section{Notation and Preliminaries}
Let $\mathcal H$ denotes the complex upper-half plane. For a complex number $\tau,$ we use the notation  $e(\tau):=e^{2\pi i \tau}.$
 The full modular group $SL_2(\mathbb{Z})$ and the congruence subgroup $\Gamma_0(N)$ of level $N\in\mathbb{N}$ is defined as follows;
$$
SL_2(\mathbb{Z}):= \left\{\begin{pmatrix}
a&b\\c&d
\end{pmatrix}: a,b,c,d \in \mathbb{Z}, ad-bc=1  \right\},$$
$$
\Gamma_0(N):= \left\{\begin{pmatrix}
a&b\\c&d
\end{pmatrix} \in SL_2(\mathbb{Z}):  c\equiv 0\mod N\right\}. $$

The group $SL_2(\mathbb{Z})$ acts on the complex upper half-plane
$\mathcal{H}$ %=\{\tau \in \mathbb{C}: {\Im}(\tau) >0 \}$
via fractional linear transformation as follows; 
$$ \gamma \tau:= \frac{a\tau+b}{c\tau+d}, ~ {\rm where} ~ \gamma= \begin{pmatrix}
a&b\\c&d
\end{pmatrix} \in SL_2(\mathbb{Z})~ {\rm and} ~ \tau \in \mathcal{H}.$$
For $k\in\mathbb{Z_{+}},~\gamma = \medmatrix a b c d \in \Gamma_0(4N),$ and a holomorphic function $f:\mathcal{H}\rightarrow \mathbb{C},$ define the weight $k+\frac{1}{2}$ slash operator as follows;
\begin{equation}\label{slash-half}
f{\mid}_{k+\frac{1}{2}} {\gamma} (\tau):= \left(\dfrac{c}{d}\right)\epsilon_d^{2k+1} (c\tau+d)^{-(k+\frac{1}{2})}f(\gamma \tau),
\end{equation}
 where $\left(\dfrac{c}{d}\right)$ is the Kronecker symbol and $\epsilon_d=\begin{cases}
 1 \quad  {\rm if } \quad d \equiv 1 \pmod 4, \\
 i \quad  {\rm if } \quad d \equiv 3 \pmod 4.
 \end{cases}$
 \begin{defn} A modular form of weight weight $k+\frac{1}{2}$  for $\Gamma_0(4N)$ is a complex-valued holomorphic function $f: \mathcal{H}\longrightarrow \mathbb C$ satisfying the following properties:
 \begin{enumerate}
 \item $f{\mid}_{k+\frac{1}{2}} {\gamma} (\tau)=f(\tau)\;\;\;\;{\rm for ~all~}\gamma=\begin{pmatrix}
a&b\\c&d
\end{pmatrix} \in \Gamma_0(4N).$
\item $f$ is holomorphic at the cusps of $\Gamma_0(4N)$ with the Fourier series expansion given by $$f(\tau)  = \displaystyle{\sum_{n \ge 0} a_f(n) e(n\tau)}.$$
\end{enumerate}
Further, a modular form $f$ of weight $k+\frac{1}{2}$  for $\Gamma_0(4N)$ is said to be a cusp form if it vanishes at every cusp of $\Gamma_0(4N)$. 
 \end{defn}
We denote by $M_{k+\frac{1}{2}}(\Gamma_0(4N))$ and $S_{k+\frac{1}{2}}(\Gamma_0(4N))$ the space of modular forms  and the space of cusp forms of weight $k+\frac{1}{2}$  for $\Gamma_0(4N),$ respectively. The Kohnen plus space $S^{+}_{k+\frac{1}{2}}(\Gamma_0(4N))$ is the subspace of cusp forms in $S_{k+\frac{1}{2}}(\Gamma_0(4N))$ whose $n$-th Fourier coefficient vanishes whenever $ ({-1})^k n \equiv 2, 3 {\pmod 4}.$ Kohnen using certain operators developed the theory of newforms of half-integral weight parallel to the Atkin-Lehner theory of newforms in integral weight case.  %{\bf Further, Manickam et.al.\cite{Ramki-Mani-Vasu} used the idea of Kohnen and studied the theory of newforms on $S_{k+\frac{1}{2}}(\Gamma_0(4N)).$} 
For more details on the theory of modular forms and newforms of half-integral weight, we refer to \cite{kob, Koh1, Ramki-Mani-Vasu, GShimura}. We assume the Ramanujan-Petersson conjecture for the Fourier coefficients of half-integral weight newfroms, i.e.,  for any $\epsilon >0$, we have $a_f(n) \ll_{\epsilon} n^{k/2 - 1/4 +\epsilon}. $

\section{Preparatory Lemmas}
In this section, we state some lemmas and recall some of the properties of Poincar{\' e} series which will be used in the proof of \thmref{Half wt form}. First we state the Poisson-Voronoi summation formula for half-integral weight cusp forms.\\

For  $f(\tau) = \displaystyle{\sum_{n \ge 1} a_{f}(n) n^{k/2 - 1/4} e(n\tau)} \in S_{k+\frac{1}{2}}(4N)$ and a smooth function $G(x)$ with compact support in $(0, \infty),$ the Poisson-Voronoi summation formula is given by the following lemma.

\begin{lem}
For any positive integers $c$ and $a$ with $\gcd(a, c) = 1,$ we have
\begin{equation}
\begin{split}
\sum_{n =1}^{\infty} a_{f}(n) e\left( \frac{an}{c} \right) G(n) & = \frac{ 2 \pi i^{k+\frac{1}{2}}}{c} \left(\dfrac{c}{d}\right)\epsilon_d^{2k+1} \sum_{n =1}^{\infty} a_{f}(n) e \left( - \frac{{d}n}{c} \right) H(n),
\end{split}
\end{equation}
where $ H(n) = \int_{0}^{\infty} G(x) J_{k - \frac{1}{2}}\left( \frac{4 \pi \sqrt{nx}}{c} \right) dx,~d$ is the multiplicative inverse of $a$ modulo $c,$ and  $J_{\nu} (z)$ denotes a Bessel function of order $\nu.$% defined by $$J_\nu(z):=$$ 
\end{lem} 
\begin{proof}
For a proof, we refer to \cite[Section 5]{Duke-Iwaniec}.
\end{proof}
For every positive integer $m $, we define the $m$-th Poincar{\'e} series ${P_{m}(\tau)}$ of weight $k+\frac{1}{2}$ on the congruence subgroup $\Gamma_{0}(4N)$  by 
\begin{equation}
\begin{split}
 P_{m}(\tau) &  :=\sum_{\gamma \in \Gamma_{\infty, 4N} \setminus \Gamma_{0}(4N) }  e^{2 \pi i m \tau} \mid_{k+\frac{1}{2}} {\gamma}(\tau),   \\ & = \sum_{\gamma \in \Gamma_{\infty, 4N} \setminus \Gamma_{0}(4N) }  \left(\dfrac{c}{d}\right)\epsilon_d^{2k+1} (c\tau+d)^{-k-\frac{1}{2}}   e^{2 \pi i m \frac{a\tau+b}{c\tau+d}},   
  \end{split}
 \end{equation}
where $\tau \in \mathcal{H}$ and $\Gamma_{\infty, 4N} =  \left\{  \pm \begin{pmatrix}
1&n \\
0&1
\end{pmatrix}  \mid n \in \mathbb{Z} \right\}   \cap \Gamma_{0}(4N) $. 

It is well-known that ${P_{m}(\tau)} \in S_{k+\frac{1}{2}}(\Gamma_{0}(4N))$ for  $m \ge 1.$ The $m$-th Poincar{\'e} series ${P_{m}(\tau)}$ has a Fourier series expansion given by 
$$P_{m}(\tau)   =\displaystyle{\sum_{n =1}^{\infty} a_{P_{m}}(n) n^{k/2 - 1/4} e(n\tau) },$$ \begin{equation} \label{PMF coeff}
\begin{split}
 m^{k/2 - 1/4} a_{P_{m}}(n) &  = \delta_{m,n} + 2 \pi i^{k+1/2} \sum_{c \ge 1, 4N \mid c} c^{-1} S_{k}(m, n; c) J_{k- \frac{1}{2}}\left( \frac{4 \pi \sqrt{mn}}{c} \right),
  \end{split}
 \end{equation}
 where $J_{\nu} (z)$ denotes a Bessel function of order $\nu,$ and $S_{k}(m, n; c)$ denotes the  Kloosterman sum defined by
  \begin{equation*}
\begin{split}
  S_{k}(m, n; c) = \sum_{a\!\!\!\! \pmod c \atop \gcd(a, c) =1} \epsilon_{a}^{-(2k+1)} \left(\dfrac{c}{a}\right) e \left( \frac{m{d}+na}{c} \right),
  \end{split}
 \end{equation*}
 here $d$ denotes the multiplicative inverse of $a$ modulo $c.$
 For more details on Poincar{\'e} series we refer to \cite[Section 6]{Duke-Iwaniec}.\\
  %and $\epsilon_{a}= \begin{cases}
 %1 \quad  {\rm if } \quad a \equiv 1 \pmod 4, \\
 %i \quad  {\rm if } \quad a \equiv 3 \pmod 4.
 %\end{cases}$ \\

The Weil-Sali{\' e} bound for the Kloosterman sum is given by  \cite[pp. 2413]{LWenZhi} :
 \begin{equation}
\begin{split}
  S_{k}(m, n; c) \ll {\gcd(m, n, c)}^{1/2} d(c) c^{1/2},
   \end{split} 
 \end{equation}
 where for a positive integer $n, ~d(n)$ denotes the number of positive divisors of $n$. 
 \begin{lem}\label{bessel integral}
For any integer  $p \ge 0$ and fixed $m, b \in \mathbb{N}$, we have 
\begin{equation*}
\begin{split}
\int\limits_{0}^{\infty} G(x) J_{k-1/2}\left( \frac{4 \pi \sqrt{m(x+b)}}{c} \right) J_{l-1/2}\left( \frac{4 \pi \sqrt{nx}}{c} \right) dx   \quad  \ll  \quad  X ([Pc (Xn)^{-\frac{1}{2}}]^{p}+ n^{-\frac{p}{2}})  \\  \qquad \quad   \times  ~~{ \rm min} \left( \left(\frac{\sqrt{X}}{c} \right)^{-\frac{1}{2}},  \left(\frac{\sqrt{X}}{c} \right)^{k-1/2} \right)   { \rm min} \left( \left(\frac{\sqrt{nX}}{c} \right)^{-\frac{1}{2}},  \left(\frac{\sqrt{nX}}{c} \right)^{l-1/2} \right). \\
  \end{split}
 \end{equation*}
 \end{lem}
 
 \begin{proof}
 For a proof, we refer to \cite[Lemma 3]{LWenZhi}.  
  \end{proof}
 
 \begin{lem} \label{Partsum}
 For sufficiently large $X$ and any arbitrarily small $\epsilon > 0$, we have
 \begin{equation*}
\begin{split}
  \displaystyle{\sum_{c \le X} \frac{ d(c)}{c}} = \frac{1}{2} (\log X)^{2} + 2 \gamma \log X + O(1) \ll ({\log X})^{2}   \ll X^{\epsilon} ,
  \end{split}
 \end{equation*}
 where $\gamma$ is the Euler constant.
 \end{lem}
 
 \begin{proof}
 Proof is an application of Abel's partial summation formula and we refer to \cite[pp.70]{Apostal}.
 \end{proof}

\section{Proof of \thmref{Half wt form}}
It is sufficient to obtain the estimate of $S(f, g, b)$ when $f$ is a Poincar{\'e} series, because the space of cusp forms  $S_{k+\frac{1}{2}}(\Gamma_{0}(4N))$ is generated by the Poincar{\'e} series $\{ {P_{m}(\tau)} \mid {m \ge 1}\}$. Thus, we obtain the estimate for $S(P_{m}, g, b) =  \displaystyle{\sum_{n \ge 1} a_{P_{m}}(n+b) a_{g}(n) G(n)},$ where  $a_{P_{m}}(n)$ is the $n$-th Fourier coefficient of $m$-th Poincar{\'e} series. 
We assume that $X$  is sufficiently large  depending on $m$. 

Substitute the expression of $a_{P_{m}}(n)$  from \eqref{PMF coeff} to obtain
\begin{equation*}
\begin{split}
 & S({P_{m}, g, b} ) 
 = \displaystyle{\sum_{n \ge 1} a_{P_{m}}(n+b) a_{g}(n) G(n)},  \\ 
 &  =\!\! \frac{2 \pi i^{k+1/2}}{m^{k/2 - 1/4}} \!\displaystyle{\sum_{n \ge 1} \!\!\!a_{g}(n) G(n)} \!\!\left(\!\! \frac{1}{2 \pi i^{k+1/2} }\delta_{m,n+b} \!+ \!\!\!\!\!\!\sum_{c \ge 1, 4N \mid c}\!\!\!\! \!\!c^{-1} S_{k}(m, n+b; c) J_{k- \frac{1}{2}}\left( \frac{4 \pi \sqrt{m(n+b)}}{c} \right) \right) ,
\end{split}
\end{equation*}
which yields
\begin{equation*}
\begin{split}
|S({P_{m}, g, b} )|
& \ll | \displaystyle{ \sum_{c \ge 1, 4N \mid c} c^{-1}  \sum_{a \pmod c \atop \gcd(a, c) =1} \epsilon_{a}^{-(2k+1)} \left(\dfrac{c}{a}\right) e \left( \frac{m{d} + ba}{c} \right)} \\
& \qquad \times  \sum_{n \ge 1}    a_{g}(n) G(n) e \left( \frac{na}{c} \right) J_{k- \frac{1}{2}}\left( \frac{4 \pi \sqrt{m(n+b)}}{c} \right)|.
\end{split}
\end{equation*}
Now, we apply Poisson-Voronoi summation formula to obtain
\begin{equation*}
\begin{split}
|S({P_{m}, g, b} )|
& \ll  |\displaystyle{ \sum_{c \ge 1, 4N \mid c} c^{-1}  \sum_{a \pmod c \atop \gcd(a, c) =1} \epsilon_{a}^{-(2k+1)} \left(\dfrac{c}{a}\right) e \left( \frac{m{d} + ba}{c} \right)} \times  \frac{ 2 \pi i^{k+\frac{1}{2}}}{c} \left(\dfrac{c}{d}\right)\epsilon_d^{2k+1}\\
& 
 \times \sum_{n \ge 1}    a_{g}(n) e \left( \frac{-n {d}}{c} \right) \int_{0}^{\infty} G(x) J_{k- \frac{1}{2}}\left( \frac{4 \pi \sqrt{m(x+b)}}{c} \right) J_{l- \frac{1}{2}}\left( \frac{4 \pi \sqrt{nx}}{c}\right) dx|.
\end{split}
\end{equation*}
\begin{equation}\label{sum-est}
\begin{split}
|S({P_{m}, g, b} )|
& \ll  |\displaystyle{ \sum_{c \ge 1, 4N \mid c} c^{-2}  }\sum_{n \ge 1}    a_{g}(n) S_{k} (m-n, b; c) \\ 
& \qquad \qquad \times  \int_{0}^{\infty} G(x) J_{k- \frac{1}{2}}\left( \frac{4 \pi \sqrt{m(x+b)}}{c} \right) J_{l- \frac{1}{2}}\left( \frac{4 \pi \sqrt{nx}}{c}\right) dx|.
\end{split}
\end{equation}
Without loss of generality, we may assume $n \ll X^{A}$, for a fixed large constant $A > 0$.  We now break the sum in \eqref{sum-est} into three parts as follows;\\
 \noindent {\bf Part   I:}  $n \ll X^{4\epsilon}$.\\
  \noindent {\bf Part II:}  $n \gg X^{4\epsilon}$ and $Pc < (nX)^{1/2} X^{-\epsilon}$.\\
 \noindent  {\bf Part III:}  $n \gg X^{4\epsilon}$  and $Pc \ge (nX)^{1/2} X^{-\epsilon}$.\\
 
\noindent {\bf Estimate for Part I:}  In this case, the contribution  for the sum \eqref{sum-est} denoted by $S_1,$ is at most $X^{\frac{3}{4}+\epsilon}$ which is obtained as follows:
\begin{equation*}
\begin{split}
S_{1} & =  \displaystyle{ \sum_{c \ge 1, 4N \mid c} c^{-2}  }\sum_{n \ll X^{4\epsilon}}    a_{g}(n) S_{k} (m-n, b; c)  \\
& \quad \quad  \times  \int_{0}^{\infty} G(x) J_{k- \frac{1}{2}}\left( \frac{4 \pi \sqrt{m(x+b)}}{c} \right) J_{l- \frac{1}{2}}\left( \frac{4 \pi \sqrt{nx}}{c}\right) dx.
\end{split}
\end{equation*}

We apply the Weil-Sali{\'e} bound for the Kloosterman sum $S_{k} (m-n, b; c)$ and the Ramanujan-Petersson bound for the Fourier coefficients $a_g(n),$ and then use \lemref{bessel integral} (with $p = 0$);
\begin{equation*}
\begin{split}
|S_{1}| & \ll \displaystyle{ \sum_{c \ge 1, 4N \mid c} c^{-2+1/2} d(c) }  \sum_{n \ll X^{4\epsilon}}    n^{\epsilon}
X ([Pc (Xn)^{-1/2}]^{p}+ n^{-p/2})  \\ 
&  \qquad  \times  { \rm min}( (\sqrt{X}/c)^{-1/2},  (\sqrt{X}/c)^{k-1/2})     \times { \rm min}( (\sqrt{nX}/c)^{-1/2},  (\sqrt{nX}/c)^{l-1/2}),
\end{split}
\end{equation*}
\begin{equation*}
\begin{split}
|S_{1}|
& \ll X^{1 + \epsilon}  \displaystyle{ \sum_{c \ge 1, 4N \mid c} c^{-3/2} d(c) } \quad   { \rm min} ( (\sqrt{X}/c)^{-1/2},  (\sqrt{X}/c)^{k-1/2}),   \\
& \ll X^{1 + \epsilon} \left( \displaystyle{ \sum_{c \le \sqrt{X}} c^{-3/2} d(c) (\sqrt{X}/c)^{-1/2} }  +  \displaystyle{ \sum_{c \ge \sqrt{X}} c^{-3/2} d(c) (\sqrt{X}/c)^{k-1/2} }    \right) ,\\
& \ll X^{1 + \epsilon} \left( \displaystyle{ X^{-\frac{1}{4}} \sum_{c \le \sqrt{X}} c^{-1} d(c) } + \displaystyle{ \sum_{c \ge \sqrt{X}} c^{-3/2} d(c)  }   \right), \\
& \ll X^{1 + \epsilon} \left( \displaystyle{ X^{-\frac{1}{4}} \sum_{c \le \sqrt{X}} c^{-1} d(c) } + \displaystyle{ \sum_{c  = 1}^{\infty} c^{-3/2} d(c)  }  \right)  
 \ll X^{\frac{3}{4} + \epsilon},
\end{split}
\end{equation*}
here we have used \lemref{Partsum} and partial summation formula to obtain the estimate for first sum, and the last sum is an absolutely convergent series.\\

\noindent {\bf Estimate for Part II:}  In this case, the integral in \lemref{bessel integral} is of order $O(X^{-p \epsilon})$. Therefore, by choosing sufficiently large $p$, we see that the integral is  of order $O(X^{-A}).$ Hence  the contribution for the sum in \eqref{sum-est}  is negligible.\\

\noindent {\bf Estimate for Part III:}  In this case, we again decompose the sum
\begin{equation} \label{sum-est1}
\begin{split}
&   \displaystyle{ \sum_{c \ge 1, 4N \mid c} c^{-2}  }\sum_{n \gg X^{4\epsilon} \atop Pc \ge (nX)^{1/2} X^{-\epsilon}}    a_{g}(n) S_{k} (m-n, b; c)  \\
& \quad \quad  \times  \int_{0}^{\infty} G(x) J_{k- \frac{1}{2}}\left( \frac{4 \pi \sqrt{m(x+b)}}{c} \right) J_{l- \frac{1}{2}}\left( \frac{4 \pi \sqrt{nx}}{c}\right) dx
\end{split}
\end{equation}
 into sub-sums of type $M\le n \le 2M$ using the dyadic division method % (and take $p =0$ in estimation.). 
 and break the sum over $c$ into the following two parts; \\
 \noindent{\bf Part (a):} $c > \sqrt{2MX}$.\\
\noindent{\bf Part (b):}  $ \sqrt{MX} \frac{ X^{-\epsilon}}{P} \le c \le \sqrt{2MX}$.\\

\noindent {\bf Estimate for Part (a):} In this case, the contribution for the sum, denoted by $S_a,$ is  at most $X^{\frac{3}{4}+ \epsilon}$ which is obtained as follows:
 \begin{equation*}
\begin{split}
S_{a} & =  \log X \underset{X^{4 \epsilon} \le M \le X^{A}}{ \rm max} \displaystyle{ \sum_{c > \sqrt{2MX} \atop  4N \mid c} c^{-2}  }\sum_{M \le n \le 2M}   | a_{g}(n) S_{k} (m-n, b; c) |  \\ 
&  \quad  \times  \left|\int_{0}^{\infty} G(x) J_{k- \frac{1}{2}}\left( \frac{4 \pi \sqrt{m(x+b)}}{c} \right) J_{l- \frac{1}{2}}\left( \frac{4 \pi \sqrt{nx}}{c}\right)  dx \right|.
\end{split}
\end{equation*}
Now, apply the Weil-Sali{\' e} bound for the Kloosterman sum, Ramanujan-Petersson bound for the Fourier coefficients $a_g(n),$ and then use \lemref{bessel integral} to obtain 
%\begin{equation*}
%\begin{split}
\begin{eqnarray*}
S_{a} & \ll  & \log X \underset{X^{4 \epsilon} \le M \le X^{A}}{ \rm max} \displaystyle{ \sum_{c > \sqrt{2MX} \atop  4N \mid c} c^{-2 + 1/2} d(c)  }\sum_{M \le n \le 2M}    n^{\epsilon}  \quad X ([Pc (Xn)^{-1/2}]^{p}+ n^{-p/2})  \\ 
& \qquad  \times &  { \rm min} \left( \left(\frac{\sqrt{X}}{c} \right)^{-\frac{1}{2}},  \left(\frac{\sqrt{X}}{c} \right)^{k-1/2} \right)   { \rm min} \left( \left(\frac{\sqrt{MX}}{c} \right)^{-\frac{1}{2}},  \left(\frac{\sqrt{MX}}{c} \right)^{l-1/2} \right) ,\\
 & \ll   & X^{1+ \epsilon} \underset{X^{4 \epsilon} \le M \le X^{A}}{ \rm max}  \left(M ^{1 + \epsilon} \displaystyle{ \sum_{c > \sqrt{2MX} } c^{-3/2} d(c) }  (\sqrt{X}/c)^{k-1/2} \times (\sqrt{MX}/c)^{l-1/2} \right).
  \end{eqnarray*}
  Finally (by taking $p =0$), we obtain
 \begin{eqnarray*}
S_{a} & \ll   & X^{1+ \epsilon} \underset{X^{4 \epsilon} \le M \le X^{A}}{ \rm max}  \left(M ^{1 + \epsilon} \displaystyle{ \sum_{c > \sqrt{2MX} } c^{-3/2} d(c) }  (\sqrt{X}/c)^{k-1/2} \right) ,\\
 & \ll  &  X^{1+ \epsilon} \underset{X^{4 \epsilon} \le M \le X^{A}}{ \rm max}  \left(M ^{1 + \epsilon} X^{\frac{k-1/2}{2}} \displaystyle{ \sum_{c > \sqrt{2MX} } c^{ -k -1} d(c) }  \right), \\
 & \ll  &  X^{1+ \epsilon} \underset{X^{4 \epsilon} \le M \le X^{A}}{ \rm max}  \left(M ^{1 + \epsilon} X^{\frac{k-1/2}{2}} \int_{x=\sqrt{2MX}}^\infty x^{-k-1+\epsilon}dx \right) ,\\
 & \ll  &   X^{1+ \epsilon} \underset{X^{4 \epsilon} \le M \le X^{A}}{ \rm max}  \left(M ^{1 + \epsilon} X^{\frac{k-1/2}{2}}  (MX)^{-k/2 + \epsilon} \right) ,\\
 & \ll &   X^{1+ \epsilon} \underset{X^{4 \epsilon} \le M \le X^{A}}{ \rm max}  \left(M ^{1 -k/2 + \epsilon} X^{-\frac{1}{4} + \epsilon} \right)  \ll  X^{\frac{3}{4} + \epsilon}. 
 \end{eqnarray*}
 \noindent{\bf Estimate for Part (b):} $ \sqrt{MX} \frac{ X^{-\epsilon}}{P} \le c \le \sqrt{2MX}$. 

  If $\sqrt{X} \ge \sqrt{MX} \frac{ X^{-\epsilon}}{P}, $ then we have  $M \le P^{2} X^{2 \epsilon}$. In this case, the contribution, denoted by $S_{b,1},$  is at most $X^{\frac{3}{4} + \epsilon} P^{\frac{3}{2}}$ which is obtained as follows:   
 \begin{equation*}
\begin{split}
S_{b,1} & =  \log X \underset{X^{4 \epsilon} \le M \le P^{2} X^{2 \epsilon}}{ \rm max} \displaystyle{ \sum_{\sqrt{MX} \frac{ X^{-\epsilon}}{P} \le c \le \sqrt{2MX} \atop  4N \mid c} c^{-2}  }\sum_{M \le n \le 2M}   | a_{g}(n) S_{k} (m-n, b; c) |  \\ 
& \qquad \qquad  \qquad  \qquad  \times  \left|\int_{0}^{\infty} G(x) J_{k- \frac{1}{2}}\left( \frac{4 \pi \sqrt{m(x+b)}}{c} \right) J_{l- \frac{1}{2}}\left( \frac{4 \pi \sqrt{nx}}{c}\right)  dx \right|.
\end{split}
\end{equation*}
Now, apply the Weil-Sali{\' e} bound for the Kloosterman sum, Ramanujan-Petersson bound for the Fourier coefficients $a_g(n),$ and then use \lemref{bessel integral} to obtain 
\begin{equation*}
\begin{split}
S_{b,1} & \ll   \log X \underset{X^{4 \epsilon} \le M \le P^{2} X^{2 \epsilon}}{ \rm max} \displaystyle{ \sum_{\sqrt{MX} \frac{ X^{-\epsilon}}{P} \le c \le \sqrt{2MX} \atop  4N \mid c}  c^{-2 + 1/2} d(c)  }\sum_{M \le n \le 2M}    n^{\epsilon}  \quad X ([Pc (Xn)^{-1/2}]^{p}+ n^{-p/2})  \\ 
& \qquad  \times  { \rm min} \left( \left(\frac{\sqrt{X}}{c} \right)^{-\frac{1}{2}},  \left(\frac{\sqrt{X}}{c} \right)^{k-1/2} \right)   { \rm min} \left( \left(\frac{\sqrt{MX}}{c} \right)^{-\frac{1}{2}},  \left(\frac{\sqrt{MX}}{c} \right)^{l-1/2} \right) ,\\
 & \ll   X^{1+ \epsilon}  \underset{X^{4 \epsilon} \le M \le P^{2} X^{2 \epsilon}}{ \rm max} M ^{1 + \epsilon} \begin{cases}
    \displaystyle{ \sum_{\sqrt{MX} \frac{ X^{-\epsilon}}{P} \le c \le \sqrt{2MX} \atop  4N \mid c} } c^{-3/2} d(c)  \quad    { \rm min} \left( \left(\frac{\sqrt{X}}{c} \right)^{-\frac{1}{2}},  \left(\frac{\sqrt{X}}{c} \right)^{k-1/2} \right)  \\
  \qquad  \qquad \qquad   \qquad  \qquad \qquad  \times \left( (\sqrt{MX}/c)^{-1/2} \right), \\
\end{cases} \\
 & \ll   X^{1+ \epsilon}\underset{X^{4 \epsilon} \le M \le P^{2} X^{2 \epsilon}}{ \rm max}     M^{1+\epsilon} \times \begin{cases} \displaystyle{ \sum_{\sqrt{X}  \le c \le \sqrt{2MX} \atop 4N \mid c}  c^{-3/2} d(c) }  (\sqrt{X}/c)^{k-1/2} \times (\sqrt{MX}/c)^{-1/2}  \quad + \\    \displaystyle{ \sum_{\sqrt{MX} \frac{ X^{-\epsilon}}{P}  \le c \le \sqrt{X} \atop 4N \mid c}  c^{-3/2} d(c) }  (\sqrt{X}/c)^{-1/2} \times (\sqrt{MX}/c)^{-1/2}, 
 \end{cases} \\
\end{split}
\end{equation*}
\begin{equation*}
\begin{split}
& \ll   X^{1+ \epsilon}\underset{X^{4 \epsilon} \le M \le P^{2} X^{2 \epsilon}}{ \rm max}     M^{1+\epsilon} \times \begin{cases}  (MX)^{-1/4}  \displaystyle{ \sum_{\sqrt{X}  \le c \le \sqrt{2MX} \atop 4N \mid c}  c^{-1} d(c) }    \quad  + \\    X^{-1/2 + 1/4} M^{-1/4} \displaystyle{ \sum_{\sqrt{MX} \frac{ X^{-\epsilon}}{P}  \le c \le \sqrt{X} \atop 4N \mid c}  \frac{d(c)}{c} } ,
 \end{cases} \\
& \ll   X^{1+ \epsilon}\underset{X^{4 \epsilon} \le M \le P^{2} X^{2 \epsilon}}{ \rm max}   \left(  M^{1+\epsilon} \times  
 (MX)^{-1/4}  \displaystyle{ \sum_{\sqrt{MX} \frac{ X^{-\epsilon}}{P}  \le c \le \sqrt{2MX} \atop 4N \mid c}  \frac{d(c)}{c} } 
\right),\\
& \ll  \underset{X^{4 \epsilon} \le M \le P^{2} X^{2 \epsilon}}{ \rm max} \left( M^{ \frac{3}{4}} X^{\frac{3}{4} + \epsilon} \right)  \ll  X^{\frac{3}{4} + \epsilon} P^{\frac{3}{2}}.
\end{split}
\end{equation*}
 If $\sqrt{X} \le \sqrt{MX} \frac{ X^{-\epsilon}}{P} $, then we have  $M \ge P^{2} X^{2 \epsilon}$. In this case, the contribution, denoted by $S_{b,2},$  is at most $X^{\frac{3}{4} + \epsilon} P^{\frac{3}{2}}$ which is obtained as follows (by taking $p =0$):
  \begin{equation*}
\begin{split}
S_{b,2} & =  \log X \underset{  P^{2} X^{2 \epsilon} \le M \le X^{A}}{ \rm max} \displaystyle{ \sum_{\sqrt{MX} \frac{ X^{-\epsilon}}{P} \le c \le \sqrt{2MX} \atop  4N \mid c} c^{-2}  }\sum_{M \le n \le 2M}   | a_{g}(n) S_{k} (m-n, b; c) |  \\ 
&  \quad  \times  \left|\int_{0}^{\infty} G(x) J_{k- \frac{1}{2}}\left( \frac{4 \pi \sqrt{m(x+b)}}{c} \right) J_{l- \frac{1}{2}}\left( \frac{4 \pi \sqrt{nx}}{c}\right)  dx \right|.
\end{split}
\end{equation*}
Now, apply the Weil-Sali{\' e} bound for the Kloosterman sum, Ramanujan-Petersson bound for the Fourier coefficients $a_g(n),$ and then use \lemref{bessel integral} to obtain 
\begin{equation*}
\begin{split}
S_{b,2} & =  \log X \underset{  P^{2} X^{2 \epsilon} \le M \le X^{A}}{ \rm max} \displaystyle{ \sum_{\sqrt{MX} \frac{ X^{-\epsilon}}{P} \le c \le \sqrt{2MX} \atop  4N \mid c} c^{-2}  }\sum_{M \le n \le 2M}   | n^{\epsilon} c^{1/2} d(c) | \times   X ([Pc (Xn)^{-1/2}]+ n^{-p/2})   \\ 
&  \qquad   \times  { \rm min}( (\sqrt{X}/c)^{-1/2},  (\sqrt{X}/c)^{k-1/2})    \times { \rm min}( (\sqrt{nX}/c)^{-1/2},  (\sqrt{nX}/c)^{l-1/2}), \\
& \ll   X^{1+ \epsilon} \underset{  P^{2} X^{2 \epsilon} \le M \le X^{A}}{ \rm max} M^{1+ \epsilon} \displaystyle{ \sum_{\sqrt{MX} \frac{ X^{-\epsilon}}{P} \le c \le \sqrt{2MX} \atop  4N \mid c}}   c^{-3/2} d(c)   (\sqrt{X}/c)^{k-1/2} \times (\sqrt{MX}/c)^{-1/2}.  \\ 
\end{split}
\end{equation*}

\noindent
Since $k \ge 2$ and  $ \sqrt{MX} \frac{ X^{-\epsilon}}{P} \le c$ $\left(i.e., \frac{ \sqrt{X}}{c} \le \frac{ P X^{\epsilon}}{\sqrt{M}}  \le 1\right) ,$ therefore
 \begin{equation*}
\begin{split}
S_{b,2} & \ll   X^{1+ \epsilon} \underset{  P^{2} X^{2 \epsilon} \le M \le X^{A}}{ \rm max} M^{1+ \epsilon}\displaystyle{ \sum_{\sqrt{MX} \frac{ X^{-\epsilon}}{P} \le c \le \sqrt{2MX} \atop  4N \mid c}}   c^{-3/2} d(c)\left( \frac{ PX^{\epsilon}}{\sqrt{M}} \right)^{3/2}    \times (\sqrt{MX}/c)^{-1/2},  \\  
& \ll   X^{1+ \epsilon} P^{\frac{3}{2}} X^{-\frac{1}{4} + \frac{3}{2} \epsilon} \underset{  P^{2} X^{2 \epsilon} \le M \le X^{A}}{ \rm max} M^{1+ \epsilon} \displaystyle{ \sum_{\sqrt{MX} \frac{ X^{-\epsilon}}{P} \le c \le \sqrt{2MX} \atop  4N \mid c}}  M^{-\frac{3}{4} - \frac{1}{4}}  c^{-1} d(c),  \\ & \ll X^{\frac{3}{4}+ \epsilon} P^{\frac{3}{2}}  \underset{  P^{2} X^{2 \epsilon} \le M \le X^{A}}{ \rm max} \quad  (M)^{ \epsilon} \log(MX)  \qquad  \ll X^{\frac{3}{4} + \epsilon} P^{\frac{3}{2}} X^{A \epsilon}   \ll \qquad   X^{\frac{3}{4} + \epsilon} P^{\frac{3}{2}}.
\end{split}
\end{equation*}

Finally, all these estimates give us the required result.

\begin{acknowledgements}
The authors would like to thank Prof. B. Ramakrishnan for support and encouragement. The second author would like to thank HRI, Prayagraj for providing financial support through Infosys grant. Finally, the authors thank the referee for careful reading of the paper and many helpful suggestions. %We thank the referee for meticulously going through the manuscript and making numerous suggestions which improved the revised version.
\end{acknowledgements}

\end{document}